\documentclass[psamsfonts]{amsart}
\def\Mfin{M_{F,\mathrm{fin}}}
\def\Minf{M_{F,\mathrm{inf}}}

\usepackage[all]{xy}

\theoremstyle{plain}
\newtheorem{thm}{Theorem}[section]

\newtheorem{lemma}[thm]{Lemma}

\newtheorem{problem}[thm]{Problem}
\newtheorem{dfn}[thm]{Definition}

\theoremstyle{remark}
\newtheorem{rem}[thm]{Remark}

\begin{document}
\date{\today}

\title{Two variable polynomial congruences and capacity theory }

\author[Chinburg]{T. Chinburg}
\address{T. Chinburg, Dept. of Mathematics\\ Univ. of Pennsylvania \\ Philadelphia, PA 19104, USA}
\email{ted@math.upenn.edu}
\thanks{T. C. was partially supported by  NSF SaTC grants No. 1513671 and 1701785.}

\author[Hemenway Falk]{B. Hemenway Falk}
\address{B. Hemenway Falk, Dept. of Computer Science\\Univ. of Pennsylvania \\ Philadelphia, PA 19104, USA}\email{fbrett@cis.upenn.edu}
\thanks{B. H-F. was supported by NSF SaTC grant No. 1513671.}

\author[Heninger]{N. Heninger}
\address{N. Heninger, Dept. of Computer Science and Engineering\\
U. C. S. D.\\
La Jolla, CA 92093}
\email{nadia@cs.ucsd.edu}
\thanks{N. H. was supported by NSF SaTC grants No. 1513671 and 1651344.}

\author[Scherr]{Z. Scherr}
\address{Z. Scherr, Dept. of Math and Computer Science\\Susquehanna University\\
Selinsgrove, PA 17870}
\email{scherr@susqu.edu}
\thanks{Z. Scherr was partially supported by NSF SaTC grant  No. 1513671}

\maketitle


\section{Introduction}
\label{s:intro}

Coppersmith's method~\cite{JC:Coppersmith97} uses lattice basis reduction to find small solutions of polynomial congruences.  This method and its variants have been used to solve a number of problems across cryptography, including attacks against low public exponent RSA~\cite{JC:Coppersmith97}, demonstrating the insecurity of small private exponent RSA~\cite{EC:BonDur99}, factoring with partial knowledge~\cite{JC:Coppersmith97}, and the approximate integer common divisor problem~\cite{HowgraveGraham01,EC:vGHV10,CohnHeninger13}.

This paper is the second in a series relating Coppersmith's method to adelic capacity theory.  In the most common approach to Coppersmith's method, which is the perspective we adopt in this paper, one constructs an auxiliary polynomial that is guaranteed by construction to have the desired solutions as roots.  Using adelic capacity theory, we showed in our first paper that in the univariate case, Coppersmith's constructive bounds are tight: Above the bound, no auxiliary polynomial of the form constructed in the algorithm can exist.

Coppersmith's method can also be applied to find solutions to multivariate polynomials or systems of polynomials.  Unlike the univariate case, which is a fully rigorous method, the method used in the existing cryptanalytic literature to address the multivariate case is heuristic.  In order to solve an $m$-variable system, one searches for $m$ (or more) suitable auxiliary polynomials in an explicitly constructed lattice, and then solves the system of auxiliary polynomials to find the possible roots.  In order for this method to work, one needs to find $m$ suitable algebraically independent polynomials constructed through the lattice.  The existing constructions are unable to guarantee the algebraic independence of multiple auxiliary polynomials, and thus the applications of this method all rely on a heuristic assumption of algebraic independence.

In this paper, we apply adelic capacity theory to two-variable linear polynomial congruences. This is the simplest case involving multivariate polynomials, and it includes the hidden number problem and ring learning with errors as special cases.  The analysis turns out to already be quite involved, and we cannot apply existing results from adelic capacity theory in a black-box way.

It is always possible to find at least one auxiliary function that is linear from the construction in Coppersmith's method.  We show that this function can be used to determine rigorously whether Coppersmith's method can succeed.    That is, we show that one can use capacity theory  to determine from the first auxiliary function whether there will be a second function that is algebraically independent of the first.   This is because the zero locus of the first function is an affine line, to which one can apply the work on capacity theory by Cantor \cite{Cantor} and Rumely \cite{Rumely}.  As a consequence of this approach, we will show that the heuristic assumption of algebraic independence does not hold in general for all problem instances. In particular, we give an infinite family of examples for which there can be no pair of  algebraically independent functions of any degree in Coppersmith's method. However, we have a method for determining rigorously whether such a pair exists in a given case.  We also give an infinite family of examples for which such a pair does exist.

If one is looking for small integral solutions of linear polynomial congruences in a particular number field, one can apply lattice techniques directly, without constructing auxiliary functions.  Coppersmith's method pertains to finding all such solutions in all number fields, i.e. in the ring of all algebraic integers.   In the case of homogenous congruences, one solution produces infinitely many by multiplying all the variables by an arbitrary root of unity.  Thus in this case, there are either infinitely many solutions in the ring of all algebraic integers, or no solutions at all.  For this reason, if
one can use capacity theoretic arguments to show that there are only finitely many solutions, one knows that in fact there are no solutions at all.  This leads  in \S \ref{s:solutioncounting}  and \S \ref{s:bounds}
to strong bounds on the number of solutions of inhomogeneous congruences as well.  In particular, we show in \S \ref{s:bounds} how this approach leads to a computable sufficient criterion for  there to exist at most one solution in any number field to a hidden number problem involving two linear congruences.  

Our methods amount to giving effective
upper and lower bounds to various finite morphism capacities in multivariable capacity theory (see \cite{FiniteMorphism}).
This is the first time to our knowledge that multivariable capacity theory has been applied to cryptography.
For a discussion of how one variable capacity theory pertains to Coppersmith's method, see \cite{AC:CHHS16}.

\subsection{The Hidden Number Problem and Ring Learning with Errors}
\label{s:hiddennumber}

In cryptographic applications, the hidden number problem is defined over the integers as follows.  In the usual formulation, there is a public integer modulus $n$ and a secret integer $s$.  For a given non-negative integer $c_i$ less than $n$, one can compute the remainder $b_i = c_i s \bmod n$ as a positive integer less than $n$.  Let $d_i$ be the integer defined by the $\ell$ most significant bits of $b_i$, and let $x_i=b_i-d_i$.  In the hidden number problem, one is given many samples $\{ (c_i, d_i) \}_{i=0}^m$ and the problem is to compute the secret integer $s$ mod $n$ from these samples.

To put this problem into the framework we consider in this paper, note that each sample satisfies the linear relation
\[
x_i - c_i s + d_i \equiv 0 \bmod n
\]
For each relation, the $x_i$ are unknown and small, and the value $s$ is unknown.  Suppose $c_0$ is relatively prime to $n$, so that $c_0 c'_0 \equiv 1$ mod $n$ for a readily computable integer $c'_0$.  The above congruence for $i = 0$ then gives
\begin{equation}
\label{eq:seq}
s \equiv c'_0 (x_0 + d_0)   \bmod n
\end{equation}
Substituting this into the congruences for $i = 1, \ldots , m$ then gives a new system of congruences
\begin{equation}
\label{eq:newsystem}
x_i + t_i x_0 + a_i \equiv 0 \bmod n\quad \mathrm{for}\quad 1 \le i \le m
\end{equation}
in small unknowns $x_i$ and $x_0$, where $t_i = -c_i  c'_0$ and $a_i = d_i - c_i c'_0 d_0$ are computable from the given data.  Because of (\ref{eq:seq}), we can reformulate the problem of finding $s$ mod $n$ as finding a solution $\{x_i\}_{i = 0}^m$ to the system of congruences (\ref{eq:newsystem}) with appropriate size bounds on all of the $x_i$.

The ``usual'' method used to solve this problem comes from Boneh and Venkatesan~\cite{C:BonVen96}, and consists of solving a closest vector problem where the solution vector corresponds to the desired solution to the problem.  In this paper we consider a dual construction, corresponding to Howgrave-Graham's reformulation of Coppersmith's method~\cite{HowgraveGraham97,Coppersmith01}:  Using lattice methods, we try to construct polynomials in the variables $\{x_i\}_{i = 0}^m$ which must vanish on all solutions, and whose common zero locus is finite.

Boneh and Venkatesan give bounds for which with high probability there is a unique solution when the $t_i$ are generated uniformly at random modulo $n$.  In practical applications of this method, one is dealing with fixed parameters.  In these cases one can empirically measure the probability of success~\cite{EC:AlbHen21}, but a rigorous analysis of the number of possible solutions has not been done in the literature.

In the ring learning with errors problem~\cite{EC:LyuPeiReg10}, one has a public commutative ring $R$, typically an order in the ring of integers $O_F$ of a number field $F$, and a secret $s \in R$.  The input to the problem is a set of samples $\{(a_i,b_i)\}_{i = 1}^n$ of pairs of elements of $R$ for which $b_i$ is congruent to $a_i \cdot s + e_i$ modulo a given ideal $\mathcal{I} \subset R$, where 
$e_i \in R$ is an unknown error that is small in some sense.  Typically the $e_i$ must be ``short"
relative to the complex embeddings of $R$.

\subsection{Solution Counting and Capacity Theory}
\label{s:solutioncounting}

The problem that we consider in this paper unifies both of the above problems, but we limit ourselves to the case of two samples.   As noted above, one can eliminate the unknown secret $s$ and obtain a single two-variable linear polynomial where the desired solution for both variables is ``small''.  For a given solution, one can then use the original polynomials to solve for a unique $s$ determined by that solution.

One basic question is:  How unique is $s$?  When $n = 2$, will show in \S \ref{s:bounds} that there is at most one $s$ when the capacity associated to an adelic set arising from the lattice construction is less than $1$.  

A peculiarity of this approach is that in the ring-LWE case, we can actually bound the number of solutions in \emph{all} number fields, and not just in a particular $F$.

We will now describe the three related problems we will study.  Let $O_F$ be the ring of integers of a number field $F$.  
Let $\mathcal{J}$ be a non-zero ideal of $O_F$.  We will suppose $a$ and $t$ are elements of $O_F$  such that $\mathcal{J} + t O_F = O_F$, so that $t$ projects to
a unit of $O_F/\mathcal{J}$ if $\mathcal{J} \ne O_F$.   Let $\overline{\mathbb{Z}}$ be the integral closure of $O_F$ in an algebraic closure $\overline{\mathbb{Q}}$ of $\mathbb{Q}$.

\begin{problem}
\label{prob:Integers}    For which real numbers $X, Y > 0$ is there a finite time algorithm for listing all  $x , y \in \overline{\mathbb{Z}}$ such that 
\begin{enumerate}
\item[1.] $x + ty + a \equiv 0$ mod $\mathcal{J} \overline{\mathbb{Z}}$ and
\item[2.] For every ring embedding $\lambda:\overline{\mathbb{Z}} \to \mathbb{C}$, the images $x'$ and $y'$ of $x$ and $y$ 
satisfy $|x'| \le X$ and $|y'| \le Y$.
\end{enumerate}
\end{problem}

\begin{problem}
\label{prob:algint} For which real $X, Y > 0$  are there only finitely many algebraic integers 
$x, y \in \overline{\mathbb{Z}}$ having the properties in Problem \ref{prob:Integers}?
\end{problem}

\begin{problem}
\label{prob:function}  Construct non-zero polynomials $g(x,y) \in F[x,y]$
with the following properties:
\begin{enumerate}
\item[1.] For all $x, y \in \overline{\mathbb{Z}}$ such that $x + ty + a \equiv 0$ mod $\mathcal{J} \overline{\mathbb{Z}}$ the
value $g(x,y)$ lies in $\overline{\mathbb{Z}}$.
\item[2.] Suppose $x', y' \in \mathbb{C}$ and that 
$|x'| \le X$ and $|y'| \le Y$.  Then $|\lambda(g)(x',y')| < 1$ for all embeddings $\lambda:F \to \mathbb{C}$,
where $\lambda(g)(x,y) \in \mathbb{C}[x,y]$ is the image of $g(x,y)$ under the homomorphism $F[x,y] \to \mathbb{C}[x,y]$
induced by $\lambda$.  
\end{enumerate}
\end{problem}

Note that Problem \ref{prob:Integers} is a constrained learning with errors problem with secret $s = y \in \overline{\mathbb{Z}} = R$ when one has a single 
data point $(a_1,b_1) = (t,-a)$ , the ideal $\mathcal{I}$ is $\mathcal{J}\overline{\mathbb{Z}}$ and the error $e_1$ is $-x$.   

Coppersmith's method relates these problems  in the following way.  

Suppose $\{g_i(x,y)\}_i$ is a family
of polynomials which each have the properties in Problem \ref{prob:function}.  Let $(x,y)$ be a solution
of  Problem \ref{prob:algint}. Then $g_i(x,y)$ will be an algebraic integer. Every embedding of $g_i(x,y)$
into $\mathbb{C}$ lies in $\mathbb{R}$ and has the form $\lambda(g_i)(x',y')$ for some conjugates $x' = \lambda(x)$ of $x$
and $y' = \lambda(y)$ of $y$ and some embedding $\lambda:\overline{\mathbb{Z}} \to \mathbb{C}$.  Since $|\lambda(g_i)(x',y')| < 1$ for all such $(x',y')$, the product formula (or an easy norm argument)
shows $g_i(x,y) = 0$.   

Suppose now that  the common zero locus of the family $\{g_i(x,y)\}_i$ is finite.  It follows that there are finitely
many solutions $(x,y) $ to Problem \ref{prob:algint}, and these solutions contain those of Problem 
\ref{prob:Integers}. 
If one has an algorithm for producing a family of $\{g_i(x,y)\}_i$ with all of these properties, as well
as for finding their finite set of common zeros, one has an algorithm for solving 
 Problem \ref{prob:Integers}.  
 
Suppose, to the contrary, that there are infinitely many solutions to Problem \ref{prob:algint}. 
Then the common zero locus of any family  $\{g_i(x,y)\}_i$ of the above kind cannot be finite, and 
Coppersmith's method cannot lead to a finite time algorithm to solve Problem \ref{prob:Integers}.

 We can now state our main result in qualitative terms;  a more quantitative version is given in
 Theorem \ref{thm:mainthm}.  Let $r_1(F)$ and $r_2(F)$ be the 
 number of real and complex places of $F$, and let $D_{F/\mathbb{Q}}$ be the disciminant of $F$.
 
 \begin{thm} 
\label{thm:mainthmearly}Suppose $X > 0$ and $ Y > 1/3$  satisfy the inequality
\begin{equation}
\label{eq:bigineq}
( \pi/ 2)^{3 r_2(F)}  \cdot 3^{-3 [F:\mathbb{Q}]}\cdot  |D_{F/\mathbb{Q}}|^{-3/2} \cdot \mathrm{Norm}_{F/\mathbb{Q}}(\mathcal{J}) > (XY)^{[F:\mathbb{Q}]} 
\end{equation}
There exists  
a non-zero linear function $g_1(x,y) = \tau x + \gamma y + \delta \in F[x,y]$ with the properties in Problem \ref{prob:function}.
Given any such $g_1(x,y)$, one of the following statements is true, and there is a procedure for determining which of the following alternatives hold:
\begin{enumerate}
\item[1.]   Suppose we decrease both $X$ and $Y$ by arbitrarily small amounts.  
Then there is a polynomial $g(x,y) \in F[x,y]$ for which the conditions
in  Problem \ref{prob:function} hold for which the common zero locus of $g(x,y)$
and $g_1(x,y)$ is finite.  Such a $g(x,y)$ leads  to a solution of Problem
\ref{prob:Integers}  for the new values of $X$ and $Y$, and there are only finitely many solutions to Problem \ref{prob:algint} for these values.   
\item[2.] Suppose we increase both $X$ and $Y$ by arbitrarily small amounts.  Then for these new values of $X$ and $Y$, 
all polynomials $g(x,y) \in F[x,y]$
of any degree having the properties in Problem \ref{prob:function} are divisible by $g_1(x,y)$.  There are infinitely
many solutions to Problem \ref{prob:algint} for the new values of $X$ and $Y$, and thus Coppersmith's method
in the above form cannot be used to solve Problem \ref{prob:Integers}.
\end{enumerate}
One of these alternatives must hold, and they are not mutually exclusive.
\end{thm}

We show in Theorem \ref{thm:mainthm} that the case in which both alternatives (1) and (2) occur is when a certain adelic capacity is exactly equal to $1$.  
We  construct in Theorem \ref{thm:propexamples}  infinitely many examples in which option (1) occurs but (2) does not, and infinitely many other examples in which option (2) occurs but (1) does not.  For these examples, $F = \mathbb{Q}$, $\mathcal{J} = p \mathbb{Z}$ for a prime $p$ of increasing size and $X = Y = c \sqrt{p}$ for a fixed positive constant $c$.  We show that each of options (1) and (2) occur for a positive proportion (as $p \to \infty$) of pairs of $t$ and $a$ 
in $\mathbb{Z}/p$ for which $t$ is prime to $p$.

In \S \ref{s:bounds}, we consider bounds on the number of solutions of Problem \ref{prob:function} when case (1) occurs.  When a capacity associated to $X$ and $Y$ is sufficiently small, we show in Theorem \ref{thm:bouncer}  that there is at most one pair $(x,y)$ with the properties in Problem \ref{prob:Integers}.  This relies on the fact that in the special case when $a = 0$, multiplying both $x$ and $y$ by a root of unity leads to another solution. Therefore when $a = 0$, either one has no solutions or an infinite number.  Because
of this, if one can show there are only finitely many solutions via capacity theory when $a = 0$, there in fact can be no solutions at all.  This fact leads to another phenomenon, namely that when $a = 0$, a small solution to a linear homogeneous congruence prevents the existence of solutions which have uniformly smaller archimedean absolute values.
We state one example here:  a more general result is shown in  Theorem  \ref{thm:surprise}.

\begin{thm}
\label{thm:nonono}  Suppose $n$ is is a positive integer, $\mathcal{J} = n O_F$, $a = 0$  and $XY \le n/2$.   Suppose $(x_0,y_0)$  is a pair of algebraic integers with the properties stated for $x$ and $y$ in Problem \ref{prob:Integers}.  Assume in addition that $x_0, y_0$ and $n$ are coprime in the sense that no pair of these numbers is contained in a proper ideal of $O_F$. Then there is no non-zero pair $(x_1, y_1)$  of algebraic integers having the properties in Problem \ref{prob:Integers} for which the following is true:  
$|\lambda(x_1)| \le  |\lambda(x_0)|$ and $|\lambda(y_1)| \le  |\lambda(y_0)|$ for all embeddings $\lambda:\overline{\mathbb{Z}} \to \mathbb{C}$ with strict inequality holding for at least one of $x$ or $y$ for at least one $\lambda$.
\end{thm}

\section{Constructing one auxiliary function}

It is well known that the existence of one function of the kind in Problem \ref{prob:function} for sufficiently small positive values of $X$ and $Y$ is a  consequence of Minkowski's
theorem:

\begin{thm}
\label{thm:lin}
Suppose $X > 0$ and $ Y > 1/3$  satisfy the inequality (\ref{eq:bigineq}). 
There exists a a polynomial $g_1(x,y) = b_1 x + b_2 y  + b_3 \in \mathcal{J}^{-1}\cdot O_F[x,y]$ with the following properties:
\begin{enumerate}
\item[i.] For all embeddings $\lambda:F \to \mathbb{C}$ one has 
$$|\lambda(b_1)| < 1/(3X), \quad  |\lambda(b_2)| < 1/(3Y), \quad \mathrm{and}\quad  |\lambda(b_3)| < 1/3.$$ 
\item[ii.] 
$g_1(x,y) = 0$ for all pairs algebraic integers $(x,y)$ as
in Problem \ref{prob:algint}. 
\item[iii.]  $b_1 \ne 0$ and $g_1(x,y) \equiv b_1(x + ty + a)$ mod $O_F[x,y]$
\end{enumerate}
All such $g_1(x,y)$  have the properties in Problem \ref{prob:function}.
\end{thm}

\begin{proof}Let $\mathbb{R}_F = \mathbb{R} \otimes_{\mathbb{Q}} F = \oplus_{v \in M_\infty} F_v$ where $M_{\infty}$ 
is the set of archimedean places of $F$ .  Give $\mathbb{R}_F$ the Euclidean norm resulting from the usual Euclidean norms $| \ |_v$ 
on the $F_v$.  (Note that the normalized absolute value on $F_v$ is $|\ |_v^{[F_v:\mathbb{R}]}$.)    Let $V = \mathbb{R}_Fx + \mathbb{R}_Fy + \mathbb{R}_F$ be the real vector space of all 
polynomials of degree at most $1$ over $\mathbb{R}_F$.  We give $V$ the Euclidean inner product resulting from viewing
it as a free $\mathbb{R}_F$-module on $\{x, y, 1\}$.  Let $L \subset V $ be the $O_F$-sublattice 
\begin{equation}
\label{eq:Ldef}
L = \mathcal{J}^{-1} \cdot (x + ty + a) + O_F \cdot y + O_F.
\end{equation}
  Then $$\mathrm{covolume}(V/L) =  2^{-3r_2(F)} |D_{F/\mathbb{Q}}|^{3/2} \mathrm{Norm}_{F/\mathbb{Q}}(\mathcal{J})^{-1}$$
by \cite[Lemma 2, Chap. V.2]{Lang}. 

For $d > 0$ define $B(0,d)$ to be the set of all $\xi  = (\xi_v)_{v \in M_\infty} \in R_F = \oplus_{v \in M_\infty}F_v$ such that $|x_v| < d$ for all $v \in M_\infty$.   Consider the convex symmetric subset 
$$S(d_1, d_2, d_3) = \{r_1 x + r_2 y + r_3: r_1 \in B(0, d_1), r_2 \in B(0, d_2), r_3 \in B(0,d_3) \}.$$
Then 
$$\mathrm{vol}(S(d_1, d_2, d_3)) = (2^{r_1(F)} \pi^{ r_2(F)})^3 (d_1 d_2 d_3)^{[F:\mathbb{Q}]}.$$
Suppose
\begin{equation}
\label{eq:inequality1}
\mathrm{vol}(S)  > 2^{3[F:\mathbb{Q}]} \mathrm{colvolume}(V/L)
\end{equation} 
 Minkowski's theorem then 
guarantees that there is a non-zero 
$g_1(x,y) = b_1 x + b_2 y  + b_3\in L \cap S$.

Suppose
$(x,y) \in \overline{\mathbb{Z}}^2$ has the properties in Problem \ref{prob:algint}, so that $x + ty + a \in \mathcal{J} \overline{\mathbb{Z}}$, 
$|\lambda(x)| < X$ and $|\lambda(y)| < Y$ for all embeddings $\lambda:\overline{\mathbb{Z}} \to \mathbb{C}$.  From the definition of $L$
and that fact that $g_1$ is a polynomial in $L \cap X$, 
we find that $g_1(x,y) \in \overline{\mathbb{Z}}$ and
$$|\lambda(g_1(x,y))| <   d_1 X + d_2 Y + d_3$$
for all $\lambda$.  Thus if we choose $d_1, d_2, d_3$ such that 
\begin{equation}
\label{eq:optimal}
d_1 X + d_2 Y + d_3 = 1
\end{equation} we can conclude that $g_1(x,y) = 0$ since the norm of $g_1(x,y)$ to $\mathbb{Z}$ is an integer of absolute
value less than $1$.  The choice of $d_1, d_2, d_3 > 0$ for which (\ref{eq:optimal}) holds and $\mathrm{vol}(S)$ is maximized is
$$(d_1, d_2, d_3)= (1/(3X), 1/(3Y), 1/3)$$
leading to 
$$\mathrm{vol}(S) = (2^{r_1(F)} \pi^{ r_2(F)})^3 (d_1 d_2 d_3)^{[F:\mathbb{Q}]} =  (2^{r_1(F)} \pi^{ r_2(F)})^3 \cdot (3^{-3}/ (XY))^{[F:\mathbb{Q}]}.$$
Combining this with the Minkowski inequality (\ref{eq:inequality1}) leads to the conclusion that if $XY$ satisfies the inequality in (\ref{eq:bigineq}),
then (i) and (ii) of Theorem \ref{thm:lin} hold.

Finally, suppose $b_1 = 0$. The definition of $L$ in (\ref{eq:Ldef}) then shows that $g_1(x,y) = b_2 y + b_3$ with $b_2, b_3 \in O_F$.  However,
property (i) of Theorem \ref{thm:lin} together with our assumption that $Y > 1/3$ forces $b_2$ and $b_3$ to have all conjugates of absolute value less than $1$.  This
forces $b_2 = b_3 = 0$ as well, contradicting the fact that $g_1(x,y)$ is a non-zero polynomial.
\end{proof}

\begin{rem} There may be more than one $g_1(x,y)$ with the properties in Theorem \ref{thm:lin}.  
\end{rem}

\section{Adelic subsets of the zero locus of the first auxiliary function.}

The strategy now for studying Problem \ref{prob:algint} is to use the fact that all solutions must be
on the zero locus of the auxillary function described in Theorem \ref{thm:lin}. This zero locus is
an affine line. We will determine the adelic constraints that the Galois conjugates of a point on this line must satisfy which are
equivalent to providing a solution to Problem \ref{prob:algint}.  We then apply adelic capacity theory on the line
to determine whether or not there are infinitely many such solutions, and whether there is a second auxillary polynomial
with the right adelic properties  which is not divisible by the first one produced by Theorem \ref{thm:lin}.

Throughout this section, we fix the following notations.

\begin{dfn}
\label{def:definitions} Let $g_1(x,y) = b_1 x + b_2 y + b_3$ be a polynomial wth the properties in Theorem \ref{thm:lin}. 
Let $M_F = \Mfin \cup \Minf$ be the set of all absolute values $v$ of $F$, where $\Mfin$ (resp. $\Minf$) is the
set of finite (resp. infinite) places.  Let $\overline{F}_v$ be an algebraic closure of $F_v$.
If $v \in \Mfin$, let $| \ |_v$ be an extension to $\overline{F}_v$ of the usual normalized absolute value on $F_v$,  Let $|\mathcal{J}|_v$
in this case  be the value of $| \ |_v$ on any element of $\mathcal{J} \subset O_F$ which generates the completion of $\mathcal{J}$ at $v$.
If $v \in M_{F,\inf}$, identify $\overline{F}_v$ with $\mathbb{C}$ and let $| \ |_v$ be the usual Euclidean absolute value.  
Define a subset of $E_v$ of $\overline{F}_v$ in the following way:
\begin{enumerate}
\item[i.]  If $v$ is non-archimedean, let $E_v$ be the set of $y \in \overline{F}_v$ such that $|y|_v \le 1$, $|b_2 y + b_3|_v \le |b_1|_v$ and
$|-(b_2 y + b_3)/b_1 + ty + a |_v \le |\mathcal{J}|_v$.  
\item[ii.]  If $v$ is archimedean, let $E_v$ be the set of $y \in \overline{F}_v$ such that $|y|_v \le Y$ and $|b_2 y + b_3|_v \le |b_1|_v\cdot X$.
\end{enumerate}
\end{dfn}

\begin{lemma}
\label{lem:answers}  If $v\in \Mfin$  then $E_v$ is either empty or a disk of the form $$D(c_v, r_v) = \{y \in \overline{F}_v: |y - c_v|_v \le r_v\}$$ for some
$c_v \in F_v$ and $0 \le r_v \in \mathbb{R}$.  For all but finitely many  $v\in \Mfin$ one can take $c_v = 0$ and $r_v = 1$, in which case $E_v = D(0,1)$.  If $v$ is archimedean, then $E_v$ is either empty or the non-empty intersection of two disks in $\overline{F}_v  = \mathbb{C}$
which have centers at $0$ and at a point in $F_v$.  The adelic set $\mathcal{E} = \prod_v E_v$ has capacity relative the point $\infty$ on $\mathbb{P}^1_F$ 
equal to 
\begin{equation}
\label{eq:capequation} \gamma(\mathcal{E}) = \prod_{v \in M_F} \gamma_v(E_v)
\end{equation}
where $\gamma_v(E_v)$ is the local capacity of $E_v$ as a subset of $\mathbb{P}^1(\overline{F}_v) - \{\infty\} = \mathbb{A}^1(\overline{F}_v) = \overline{F}_v$.
One has $\gamma_v(E_v) = r_v^{[F_v:\mathbb{Q}_{p(v)}]}$ if $v$ is finite of residue characteristic $p(v)$.  If $v$ is infinite, $\gamma_v(E_v)$ is computed in Theorem \ref{thm:explicit} below.
\end{lemma}

\begin{proof} The description of $E_v$ is clear from Definition \ref{def:definitions} together with the fact that the intersection of any two non-archimedean disks in
$\overline{F}_v$ for $v \in \Mfin$ is either empty or equal to a disk.  The fact that $E_v = D(0,1)$ for all but finitely many $v \in \Mfin$ follows
from the the fact each of the three inequalities defining $E_v$ describes either $D(0,1)$ or all of $\overline{F}_v$ for all but finitely many $v$.  Hence
$\mathcal{E}$ has a well defined capacity with respect to $\infty$, and the formula (\ref{eq:capequation}) is shown in \cite[p. 366]{Rumely}.  The fact that 
$\gamma_v(E_v) = r_v^{[F_v:\mathbb{Q}_{p(v)}]}$ for $v \in \Mfin$ is shown in \cite[p. 352]{Rumely} on taking account our normalization of $| \ |_v$.   
\end{proof}

\begin{rem}
\label{rem:easyremark}
The constants $c_v$ and $r_v$ for $v \in \Mfin$ are readily computed from the coefficients of $g_1(x,y) = b_1 x + b_2 y + b_3$.  The same is true 
for the centers and radii of the two disks whose intersection is $E_v$ when $v \in \Minf$.  Thus (\ref{eq:capequation}) is readily computable from
$g_1(x,y)$ using Theorem \ref{thm:explicit}.
\end{rem}

\begin{thm} 
\label{thm:mainthm} Let $\gamma(\mathcal{E})$ be as in (\ref{eq:capequation}).
\begin{enumerate}
\item[1.] If $\gamma(\mathcal{E}) > 1$ then there are infinitely many solutions $(x,y)$ to Problem \ref{prob:algint}.  In this case, any
polynomial $g_i(x,y) \in F[x,y]$ with the properties in Problem \ref{prob:function} must be divisible
by $g_1(x,y)$ in $F[x,y]$.  In particular, the intersection of the zero loci of all polynomials 
$g_i(x,y)$ with the properties in Problem \ref{prob:function} is the zero locus of $g_1(x,y)$, which is infinite.
\item[2.] If $\gamma(\mathcal{E}) < 1$ then there are only finitely many solutions $(x,y)$ to Problem \ref{prob:algint}.
The common zero locus of the polynomials in Problem \ref{prob:function} is finite.  
\item[3.]  Suppose $\gamma(\mathcal{E}) \ne 0$.  As a function of $X> 0 $ and $Y > 1/3$, the value of $\gamma(\mathcal{E})$ strictly increases when both $X$ and $Y$ are increased. 
In particular, suppose $\gamma(\mathcal{E}) = 1$ for particular values of $X$ and $Y$.  Any increase of both $X$ and $Y$ leads to the conclusions of part (1), while any 
decrease of both $X$ and $Y$ leads to the conclusions of part (2).  This establishes parts (1) and (2) of Theorem \ref{thm:mainthmearly}.
\end{enumerate}
\end{thm}

\begin{proof}   Recall that a solution $(x,y)$ to Problem \ref{prob:algint} is a pair of $x, y \in \overline{\mathbb{Z}}$  such
that $x + ty  + a \equiv 0$ mod $\mathcal{J} \overline{\mathbb{Z}}$ and all archimedean conjugates $x'$ and $y'$ of $x$
and $y'$  satisfy $|x'| \le X$ and $|y'| \le Y$.  We know that 
$$g_1(x,y) = b_1 x + b_2 y + b_3 = 0$$
for all such $(x,y)$, where $b_1 \ne 0$. So we have
\begin{equation}
\label{eq:dumb}
x = \frac{-b_2 y - b_3}{b_1}
\end{equation}
Because of (\ref{eq:dumb}), these conditions on $x$ and $y$ translate into the condition that all conjugates over $F$ of the element $y \in \overline{F}$
lie in the set $E_v \subset \overline{F}_v$ described in Definition \ref{def:definitions} for each $v \in M_F$. 
Parts (1) and (2) of the Theorem are now consequences of Cantor's results concerning Fekete-Szeg\"o theorems for the projective line (see \cite[Theorems 5.1.1 and 5.1.2]{Cantor} and \cite[Theorems 6.3.1 and 6.3.2]{Rumely}).

For part (3), note that if $\gamma(\mathcal{E}) \ne 0 $ then $\gamma_v(E_v) \ne 0$ for all $v$.  In particular,  $E_v$ cannot be empty or a single point.  Supose now that $v \in \Minf$. Then  
$E_v$ is  intersection of two closed disks, and $E_v$ has non-empty interior. Increasing $X$ and $Y$ to some $X' > X$ and $Y' > Y$ then expands both disks.  This puts
the set $E_v$ for $X$ and $Y$ into $\lambda \cdot E'_v$ for some positive real $\lambda < 1$ when $E'_v$ is the corresponding intersection of disks for $X'$ and $Y'$.  We thus have $\gamma_\infty(E_v) \le \gamma_\infty(\lambda E'_v)  = \lambda \cdot \gamma_\infty(E'_v) < \gamma_\infty(E'_v)$.  This proves that $\gamma_\infty(E_v)$ strictly increases when we increase both $X$ and $Y$, which implies part (3) of the Theorem. 
\end{proof}

We now give the formula for the capacity $\gamma_v(E_v)$ for archimedean $v$ which was referred to at the end of the statement of Lemma \ref{lem:answers}.
Let $D(a,t)$ the closed disk in $\mathbb{C}$ with center $a \in \mathbb{C}$ and radius $r \ge 0$. 

\begin{thm}
\label{thm:explicit}  Suppose $E_v$ is the intersection in $\overline{F}_v = \mathbb{C}$ of two closed disks, one of which is centered at the origin. Then there is non-zero complex number $\xi$ such that $E_v = \xi \cdot V$ where $V = D(0,r) \cap D(1,s)$ for some $r, s \ge 0 $.  One has $\gamma_v(E_v) = (|\xi| \cdot \gamma_\infty(V))^{[K_v:\mathbb{R}]}$ where $\gamma_\infty(V)$ is the classical transfinite diameter of $V$, which may computed in the following way.   If $V = \emptyset$ or $r + s = 1$ then $\gamma_\infty(V) = 0$.  If $r \ge 1 + s$ then $V = D(1,s)$ and $\gamma_\infty(V) = s$.  If $s \ge 1 + r$ then 
$V = D(0,r)$ and $\gamma_\infty(V) = r$. Otherwise, the boundaries of $D(0,r)$ and $D(1,s)$ intersect at two points $u$ and $\overline{u}$ with $u$ in the upper half plane.  Let
$\alpha \in (0,\pi)$ be the angle  between the boundary of $D(0,r)$ and the boundary of $D(1,s)$ at the intersection point $a$.  There is a unique point $\zeta$ in the upper half plane
such that
\begin{equation}
\label{eq:zetadef}
\zeta  = \left ( \frac{\overline{u} - r}{u - r} \right )^{\pi/(2\pi - \alpha)}
\end{equation}
when we compute the complex exponential using the branch of $\mathrm{log}$ with imaginary part lying in $[0,2\pi]$.  One has 
\begin{equation}
\label{eq:answer} \gamma_\infty(V) =  \frac{1}{2 \mathrm{Im}(\zeta)} \cdot \frac{\pi}{2 \pi - \alpha} \cdot |\overline{u} - u|
\end{equation}
\end{thm}

To prove this result we will need the following fact from \cite[p. 339]{Rumely} .

\begin{lemma}
\label{lem:Rumlem} (Rumely) Suppose $E$ is a connect subset of $\mathbb{C}$.  Let  $E^c$ be the complement of $E$ in 
$\mathbb{P}^1(\mathbb{C}) = \mathbb{C} \cup \{\infty\}$.  Suppose $f(z)$ is a conformal map which takes 
$E^c$ to 
the complement of the closed disc $D(0,R)$ of radius $R > 0$.  Suppose further that $\lim_{z \to \infty} f(z)/z = 1$. Then $\gamma_\infty(E) = R$.  The Green's function
$G(z,\infty:E)$ equals $\mathrm{log}|f(z)/R)|$ for $z \in E^c$ and $G(z,\infty;E) = 0$ for $z \in E$.  
\end{lemma}

\noindent {\bf Proof of Theorem \ref{thm:explicit}}
\medbreak
The first statement is clear on rotating and dilating $E_v$ appropriately, and the formula $\gamma_\infty(E_\infty) = (|\xi| \gamma_\infty(V))^{[K_v:\mathbb{R}]}$ is shown in \cite[p. 352]{Rumely}.  Since the transfinite diameter of a disk or radius $R$ is $R$, it will suffice to construct an $f(z)$ for the $E =V$ of Theorem \ref{thm:explicit} on the assumption that the boundaries of $D(0,r)$ and $D(1,s)$ intersect at the distinct points
$u$ and $\overline{u}$.  The cross ratio $w(z) = \frac{(u - z)(\overline{u} - r)}{(u - r)(\overline{u} - z)}$ has
$w(u) = 0$, $w(r) =  1$ and $w(\overline{u}) = \infty$.  Since fractional linear transformations send circles to either lines or circles, we find that
$z \to w(z)$ sends the arc formed by the boundary of $D(0,r)$ between $u$ and $\overline{u}$ to the union of the non-negative real axis with $\infty$.  Since fractional linear transformations also
preserve angles, $z \to w(z)$ sends the arc formed by the boundary of $D(1,s)$ between $u$ and $\overline{u}$ to the union of $\{\infty\}$ with the ray $L$ outward from $w(u) = 0$ which makes an angle
of $\alpha$ from the positive real axis and lies in the lower half plane.  Thus the image of $V$ under $z \to w(z)$ is the union of $\{\infty\}$ with an angular sector in $\mathbb{C}$ bounded by the non-negative real 
axis and $L$.  The complement $V^c$ is sent to set $w(V^c)$ of  non-zero complex numbers $\tau = |\tau| e^{2\pi i \theta}$ for which $0 < \theta < 2\pi - \alpha$.  

On choosing
the branch of the complex logarithm with imaginary part in $[0,2 \pi)$, we have a conformal map $\nu$ taking $w(V^c)$ to the upper half plane $H$ defined by $\nu(w) = w^{\pi/(2 \pi - \alpha)}$.   This map sends $w(\infty) = \frac{\overline{u} - r}{u - r} \in w(V^c)$ to the  point $\zeta$ in (\ref{eq:zetadef}).    We can now use the conformal automorphism $h$ of $\mathbb{P}^1(\mathbb{C})$ defined by $\nu \to h(\nu) = \frac{1}{\nu - \zeta}$ to send $\zeta$ to $h(\zeta) = \infty$ and $H = (\nu \circ w)(V^c)$ to the complement in $\mathbb{P}^1(\mathbb{C})$
 of a closed disk $D$ which has a diameter going from $0 = h(\infty)$ to $\frac{i}{\mathrm{Im}(\zeta)} = h(\mathrm{Re}(\zeta))$.  
 
 Consider now the composition
 $f(z) = c \cdot (h \circ \nu \circ w)$, where $c$ is a non-zero constant we will choose so that $\lim_{z \infty} f(z)/z = 1$.  This $f$ gives a conformal map from $V^c$ to the complement of
 the image $c D$ of $D$ by multiplication by $c$. Therefore Lemma \ref{lem:Rumlem} gives 
 \begin{equation} 
 \label{eq:final}
 \gamma_\infty(V) = \gamma_\infty(D) = \frac{|c|}{2 \mathrm{Im}(\zeta)}
 \end{equation}
 It just remains to find $c$.  Here
 $$1/h(\nu(w(z))) = \nu(w(z)) - \nu(w(\infty)) = \left ( \frac{(u - z)(\overline{u} - r)}{(u - r)(\overline{u} - z)} \right )^{\pi/(2 \pi - \alpha)} -  \left ( \frac{(\overline{u} - r)}{(u - r)} \right )^{\pi/(2 \pi - \alpha)}$$
Thus
$$z/h(\nu(w(z)))  = \zeta \cdot z \cdot \left ( \left (\frac{(1 - z^{-1} u )}{(1 - z^{-1} \overline{u})}  \right )^{\pi/(2 \pi - \alpha)} - 1 \right )$$
since $\zeta = \left ( \frac{(\overline{u} - r)}{(u - r)} \right )^{\pi/(2 \pi - \alpha)}$.  Using that 
$$(1 - z^{-1} u)/(1 - z^{-1} \overline{u}) = 1 + (\overline{u} - u) z^{-1} + \mathrm{higher \ order \ terms}$$
we find that on setting 
$$c = \lim_{z \to \infty} z/h(\nu(w(z))) = \zeta \cdot \frac{\pi}{2 \pi - \alpha} \cdot (\overline{u} - u)$$
we will have $\lim_{z \to \infty} f(z)/z =  \lim_{z \to \infty} c \cdot h(\nu(w(z)) )/z = 1$, as required.  Now (\ref{eq:final}) gives
$$\gamma_\infty(V) = \frac{1}{2 \mathrm{Im}(\zeta)} \cdot \frac{\pi}{2 \pi - \alpha} \cdot |\overline{u} - u| $$
as claimed in Theorem \ref{thm:explicit} since $|\zeta| = 1$.

\begin{rem} As a check on Theorem \ref{thm:explicit}, consider the case in which $s $ tends toward $1 +r $ from below.  Then $V$ becomes closer and closer to being all of $D(0,r)$,
and some straightforward estimates show that the formula $\gamma_\infty(V)$ in Theorem \ref{thm:explicit} tends toward $\gamma_\infty (D(0,r)) = r$ as $ s \to (1 + r)$ from below.
\end{rem}

We end this section by showing the claim made in the last sentence of the introduction that there are infinitely many examples of each
of cases 1 and 2 of Theorem \ref{thm:mainthmearly}.  We prove a more quantitative result under some additional hypotheses.

\begin{thm}
\label{thm:propexamples}  Suppose $F = \mathbb{Q}$, so that $O_F = \mathbb{Z}$.  Suppose $\mathcal{J} = p\mathbb{Z}$ for some prime $p$,
and that $X = Y = c \sqrt{p} > 1/3$ for some  fixed constant $c$ for which $2/3 > c > 0$.     
\begin{enumerate}
\item[i.]   For sufficiently large primes $p$ there is a positive proportion of pairs $(t,a) \in (\mathbb{Z}/p)^* \times (\mathbb{Z}/p)$ for which the following is true.  There is a unique polynomial
$g_1(x,y) = b_1 x + b_2 y + b_3 = (d_1x + d_2 y + d_3)/p$ with the properties in Theorem \ref{thm:lin} for which $d_i \in \mathbb{Z}$ for $i = 1, 2, 3$ and $g.c.d.(d_1, d_2) = 1$ and $d_1 > 0$.
Alternative
(1) of Theorem \ref{thm:mainthm} holds for this $g_1(x,y)$ while alternative (2) of the Theorem does not.  
\item[ii.]  For sufficiently large primes $p$ there is a positive proportion   of pairs $(t,a) \in (\mathbb{Z}/p)^* \times (\mathbb{Z}/p) $, all the statements in (i) up to the last
one are true, but now alternative (2) of Theorem \ref{thm:mainthm} holds while alternative (1) does not.  
\end{enumerate} 
\end{thm}

We will first describe the strategy of the proof.  Rather than begin by choosing $(t,a)$, we instead choose polynomials $b_1 x + b_2 y + b_3 = (d_1 x + d_2 y + d_3)/p$ that will have the properties in Lemma \ref{lem:adjust} for some pair of residue classes $t$ and $a$ mod $p\mathbb{Z}$.  In Lemma \ref{lem:adjust} we specify conditions on the $d_i$ which ensure the uniqueness of the associated pair $(t,a)$ mod $p\mathbb{Z}$.  In Lemma \ref{lem:nicer} we show that for the $t$ and $a$ which arise from the construction in Lemma \ref{lem:adjust}, any polynomial with the properties in Theorem \ref{thm:lin} for $t$ and $a$ must arise from an integer multiple of the polynomial  in Lemma \ref{lem:adjust} that gives rise to $t$ and $a$.  These steps are needed because in the end we will count the polynomials produced in Lemma \ref{lem:adjust} which lead to the capacity $\gamma(\mathcal{E})$ associated to the triple $\{b_1 x + b_2 y + b_3, t, a\}$   in Definition \ref{def:definitions}  and Lemma \ref{lem:answers} being greater than $1$ (respectively less than $1$). We compute these capacities in Lemma \ref{lem:capcomp}.  To complete the proof of Theorem \ref{thm:propexamples} we then 
produce ranges for $d_1$, $d_2$ and $d_3$ which lead to $\gamma(\mathcal{E}) > 1$ (respectively $\gamma(\mathcal{E}) < 1$). We  show that these
ranges lead to positive proportion of the possible residue classes of $(t,a)$ mod $p$  satisfying conditions (i) (resp. (ii)) of Theorem \ref{thm:propexamples}.

\begin{lemma}
\label{lem:adjust} Recall that $2/3 > c > 0$, so $0 < 3c/2 < 1$.   
Let $w, z$ be real numbers such that $0 < w \le 2$ and $0 \le z < 1$.  Consider the set $S(w,z)$ of all triples $(d_1, d_2, d_3)$ of integers such that 
\begin{equation}
\label{eq:dright}
w 3c\sqrt{p}/4 \le d_1 \le 3c \sqrt{p}/2 < \sqrt{p}; \quad 1 \le d_2 \le 3c \sqrt{p}/2 < \sqrt{p}; \quad 0 \le d_3 < z p \quad \mathrm{and} \quad \mathrm{gcd}(d_1,d_2) = 1.
\end{equation}
There is an injective map $\lambda:S(w,z) \to \mathbb{Z}/p \times \mathbb{Z}/p$ which sends $(d_1,d_2,d_3)$ to the class in $\mathbb{Z}/p \times \mathbb{Z}/p$
of $(t,a)$ when 
$t$ and $a$ are the smallest non-negative integers for which $d_1 \cdot (1, t, a) \equiv (d_1, d_2, d_3)$ mod $p$.   One has $t \not \equiv 0 \not \equiv d_1$ mod $p$.  
\end{lemma}

\begin{proof} Suppose $(d_1, d_2, d_3) \in S(w,z)$.  Since $0 < w3c/2 < w $ we have $d_1 \not \equiv 0 \not \equiv d_2$ mod $p$.  Hence $t$ and $a$ are uniquely determined by $(d_1,d_2,d_3)$ and $t \not \equiv 0$ mod $p$.  Suppose $\lambda(d_1,d_2,d_3) = \lambda(d'_1,d'_2,d'_3)$ for some $(d'_1,d'_2, d'_3) \in S(w,z)$.  Then $d_2 \equiv t d_1$ mod $p$
and $d'_2 \equiv t d'_1$ mod $p$, so $d_2 d_1' - d'_2 d_1 \equiv t d_1 d'_1 - t d'_1 d_1 \equiv 0$ mod $p$.    But $d_1, d_2, d'_1, d'_2$ all lie in the interval 
$(0,3c\sqrt{p}/2) \subset (0,\sqrt{p})$, we have $0 < d_2 d_1' < p$ and $0 < d'_2 d_1< p$.  So $d_2 d_1' \equiv d'_2 d_1$ mod $p$ forces $d_2 d_1' = d'_2 d_1$.  Now since the $d_i$ are positive and $\mathrm{g.c.d}(d_1,d_2) = \mathrm{g.c.d}(d'_1,d'_2)$, we conclude $(d_1,d_2) = (d'_1,d'_2)$.  Now $d_3 \equiv d_1 a \equiv d'_1 a \equiv d'_3$ mod $p$, so $z < 1$ forces $d_3 = d'_3$.
Hence $\lambda$ is injective.
\end{proof}

\begin{lemma}
\label{lem:nicer}  With the notations and assumptions of Lemma \ref{lem:adjust}, suppose $\lambda(d_1,d_2,d_3) = (1,t, a)$ and $z \le 3wc^2$.   Suppose $g_1(x,y) = b_1 x + b_2y + b_3 \in \mathbb{Q}[x,y]$ has the properties in Theorem \ref{thm:lin}
for $t$ and $a$.  Then $e_1 = pb_1$, $e_2  = pb_2$ and $e_3 = p b_3$ are in $ \mathbb{Z}$.   Let $q = \mathrm{g.c.d.}(e_1,e_2,e_3)$.   Then $b_1 \ne 0$ and 
$$g_1(x,y) = \mathrm{sign}(b_1) \cdot  q \cdot (d_1x + d_2 y + d_3)/p.$$
Furthermore, $(d_1x + d_2 y + d_3)/p$ also has all the properties in Theorem \ref{thm:lin} for $t$ and $a$.
\end{lemma}

\begin{proof}    By Theorem \ref{thm:lin}, 
$0 \ne p g_1(x,y) = e_1x + e_2 y + e_3 \in \mathbb{Z}[x,y]$ has $e_1 x + e_2 y + e_3 \equiv e_1(x + ty + a)$ mod $pZ[x,y]$, so $(e_1, e_2, e_3) \equiv e_1(1,t,a)$ mod $p$.  Furthermore, \begin{equation}
\label{eq:eright}
0 < |e_1| < p/(3X) = \sqrt{p}/(3c); \quad |e_2| < p/(3Y) = \sqrt{p}/(3c); \quad \mathrm{and} \quad
|e_3| < p/3.
\end{equation}   
These properties are preserved if we divide $g_1(x,y)$ by the g.c.d. of $e_1, e_2$ and $e_3$ or if we multiply $g_1(x,y)$ by $-1$.  We will assume in what follows that this has been done, so that  $\mathrm{g.c.d}(e_1,e_2,e_3) = 1$ and $e_1 > 0$.  

Since $\lambda(d_1,d_2,d_3) = (t,a)$, we have 
$d_1 (1,t,a) \equiv (d_1,d_2,d_3)$ mod $p$.  Since $(e_1, e_2, e_3) \equiv e_1(1,t,a)$ mod $p$ we conclude $e_1 d_2 - d_1 e_2 \equiv 0$ mod $p$.  However,
$$|e_1 d_2 - d_1 e_2| <  p   $$
since $|e_i| < \sqrt{p}/(3c) $ and $|d_i| \le 3c\sqrt{p}/2$ for $i = 1, 2$.
Hence $e_1 d_2 - d_1 e_2  = 0$.  Since $\mathrm{gcd}(d_1,d_2) = 1$, this forces
$(e_1,e_2) = e(d_1,d_2)$ for some $e \in \mathbb{Z}$.  
Now $w3c/2 \sqrt{p} \le d_1$ so $|e| w3c \sqrt{p}/2 \le |ed_1| = |e_1| < \sqrt{p}/(3c)$ because of (\ref{eq:eright}).
So 
\begin{equation}
\label{eq:ebound}
|e| <  \frac{2}{(3c)^2 w}.
\end{equation}

Since $d_1 \not \equiv 0$ mod $p$, $d_1 d_1' \equiv 1$ mod $p$ for some integer $d_1'$. Therefore 
$$(ed_1,ed_2, e_3) = (e_1,e_2,e_3)\equiv e_1(1,t,a) \equiv e_1 d_1' (d_1,d_2, d_3) \quad \mathrm{mod} \ p.$$
 So $e\equiv e_1 d_1'$
mod $p$ and $e_3 \equiv e_1 d_1' d_3 \equiv e d_3$ mod $p$.  Thus $(e_1,e_2,e_3) - e(d_1,d_2,d_3) \equiv (0,0,0)$ mod
$p$, where $e_1 = e d_1$ and $e_2 = e d_2$. We conclude that $$p g_1(x,y) = e_1 x + e_2 y + e_3 = e (d_1 x + d_2 y + d_3) + h$$
when $h= e_3 - e d_3 \equiv 0$ mod $p$.  Here  
$$|e d_3| = |e| \cdot |d_3| < \frac{2zp}{(3c)^2w} \le 2p/3$$
because of (\ref{eq:ebound}), $|d_3| < zp$ and the assumption that $z \le w3c^2$.  Since $|e_3| < p/3$ we get 
$|h|  = | e_3 - e d_3 | < p/3 + 2p/3 < p$, which forces $h = 0$ since $h \equiv 0$ mod $p$.  Thus $pg_1(x,y) = e_1 x + e_2 y + e_3 = e(d_1 x + d_2 y + d_3)$. Since $\mathrm{g.c.d}(e_1,e_2,e_3) = 1$
and $e_1$ and $d_1$ are positive, we must have $e = 1$ and $g_1(x,y) =  (d_1 x + d_2 y + d_3)/p$.
\end{proof}

\begin{lemma}
\label{lem:capcomp}  Let $\mathcal{E} = \prod_v E_v$ be the adelic set associated by Definition \ref{def:definitions} to \hbox{$g_1(x,y) = \frac{1}{p}(d_1 x + d_2 y + d_3)$} for some $(d_1,d_2,d_3) \in S(w,z)$ 
under the assumption that $z \le 3wc^2$ in Lemma \ref{lem:nicer}.  Define
\begin{equation}
\label{eq:deldefs}
\delta_1 = \mathrm{max}(-c,\frac{-d_1 c}{d_2} - \frac{d_3}{\sqrt{p} d_2})\quad \mathrm{and}\quad \delta_2 =   \mathrm{min}(c,\frac{d_1 c}{d_2} - \frac{d_3}{\sqrt{p} d_2})
\end{equation}
Suppose $d_1$ and $d_2$ are relatively prime.  The capacity $\gamma(\mathcal{E})$ of $\mathcal{E}$ is $0$ if $\delta_1 > \delta_2$.  Otherwise,  
\begin{equation}
\gamma(\mathcal{E}) = \prod_v \gamma_v(E_v) = d_1^{-1} \gamma_\infty(E_\infty) \ge  \frac{\sqrt{p} (\delta_2 - \delta_1)}{4d_1}
\end{equation}
where $E_\infty$ is the component of $\mathcal{E}$ at the archimedean place of $\mathbb{Q}$.
\end{lemma} 

\begin{proof}  From Definition \ref{def:definitions}, the intersection of the real line with $E_\infty$ is trivial if $\delta_1 > \delta_2$, in which case $E_\infty$ is empty and $\gamma_\infty(E_\infty) = 0 = \gamma(\mathcal{E})$.
If $\delta_1  \le \delta_2$ then $E_\infty$ intersects the real line in the interval $[\sqrt{p}\delta_1,\sqrt{p}\delta_2]$.  In this case, we have
\begin{equation}
\label{eq:infbound}
\gamma_\infty(E_\infty) \ge \gamma_{\infty}([\sqrt{p}\delta_1,\sqrt{p}\delta_2]) = \frac{\sqrt{p}}{4}(\delta_2 - \delta_1).
\end{equation}

Suppose now that $v$ is a finite place of $\mathbb{Q}$. From Definition \ref{def:definitions} and the equality $b_1 x + b_2 y + b_3  = p^{-1} (d_1 x + d_2 x + d_3)$, we see that $E_v$ is the set of $y \in \overline{\mathbb{Q}}_v$ satisfying these conditions:
\begin{enumerate}
\item[i.]  $|y|_v \le 1$
\item[ii.] $|(d_2 y + d_3)/d_1|_v \le 1$ and 
\item[iii.] $|-(d_2y +d_3)/d_1 +ty + a|_v \le |\mathcal{J}|_v$
\end{enumerate}
where $|\mathcal{J}|_v = 1$ if $v \ne p$ and $|\mathcal{J}|_p = p^{-1}$.  Let us first show (i) and (ii) imply (iii).  If $v \ne p$, this is clear from $t, a \in \mathbb{Z}$.  Suppose now that $v = p$.  We know $d_1$ is prime to $p$, and that $(d_1,d_2,d_3) \equiv d_1(1, t, a)$ mod $p$.  So
$d_2 \equiv d_1 t$ mod $p$ and $d_3 \equiv d_1 a$ mod $p$.  Thus $$|(-d_2/d_1 + t) y|_v = |d_1|_v^{-1} \cdot |-d_2 + t d_1|_v \cdot |y|_v \le p^{-1} \quad 
\mathrm{if}\quad  |y|_v \le 1$$
and $$|-d_3/d_1 + a|_v = |d_1|^{-1} \cdot |-d_3 + a d_1|_v \le p^{-1} \quad 
\mathrm{if}\quad  |y|_v \le 1$$ so (iii) is implied by (i) when $v = p$.  Thus $E_v$
is the set of $y \in \overline{\mathbb{Q}}_v$  satisfying (i) and (ii).  

Recall now that $d_1, d_2, d_3$ are non-zero integers and that by assumption 
$g.c.d.(d_1,d_2) = 1$.  Thus if $|d_1|_v < 1$ then $|d_2|_v = 1$, and otherwise $|d_1|_v = 1$.  If $|d_1|_v = 1$,
then (i) implies $|(d_2 y + d_3)/d_1|_v = |d_2 y + d_3|_v \le 1$, so (ii) holds.  If $|d_1|_v < 1$, then $|d_2|_v = 1$ so (ii) is equivalent to 
$|y - (-d_3/d_2)|_v \le |d_1/d_2|_v = |d_1|_v < 1$.  Since $|-d_3/d_2|_v = |-d_3|_v \le 1$, condition (ii) implies (i) if $|d_1|_v < 1$.  We thus
find that for all finite $v$, $E_v$ is a $v$-adic disc of radius $r_v = |d_1|_v$ around point of $\mathbb{Q}$.  

We can now calculate the capacity of $\mathcal{E} = \prod_v E_v$.  The product of the local capacities at finite places is
$$\prod_{v \ \mathrm{finite}} \gamma_v(E_v) = \prod_{v \ \mathrm{finite}} r_v = \prod_{v \ \mathrm{finite}} |d_1|_v = d_1^{-1}$$
by the product formula since $d_1$ is a positive integer.  We thus find  as in Lemma \ref{lem:answers} and (\ref{eq:infbound}) that 
$$
\gamma(\mathcal{E}) = \prod_v \gamma_v(E_v) = d_1^{-1} \gamma_\infty(E_\infty) \ge \frac{\sqrt{p}(\delta_2 - \delta_1)}{4 d_1}.
$$\end{proof}
\noindent {\bf Proof of Theorem \ref{thm:propexamples}}
\medbreak

Let $d_1 = 3c \chi_1 \sqrt{p}/2$, $d_2 = 3c \chi_2 \sqrt{p}/2$ and $d_3 = \chi_3 p$ be as in Lemma \ref{lem:nicer}, so that $w \le \chi_1 \le 1$, $0 < \chi_2 \le 1$
and $0 \le \chi_3 \le z$.  We have supposed $0 < w$, $0 < 3c/2 < 1$, $0 \le z < 1$ and $z \le 3wc^2$.    For $i = 1, 2$ let 
$$ z_i(\chi_1,\chi_2,\chi_3) = \frac{1}{6c\chi_1}  \cdot (\frac{(-1)^i\chi_1c}{\chi_2} - \frac{2\chi_3}{3c\chi_2})  = \frac{(-1)^i c}{6c\chi_2} - \frac{\chi_3}{9c^2\chi_1 \chi_2}.$$ 
Then 
\begin{equation}
\label{eq:delt2comp} \frac{\sqrt{p}}{4 d_1} \delta_2 = \frac{\delta_2}{6c \chi_1} = \mathrm{min}(\frac{c}{6c\chi_1},z_2(\chi_1,\chi_2,\chi_3)) = \mathrm{min}(\frac{1}{6\chi_1},\frac{1}{6 \chi_2} - \frac{\chi_3}{9c^2\chi_1 \chi_2})
  \end{equation}
\begin{equation}
\label{eq:delt1comp} \frac{\sqrt{p}}{4 d_1} \delta_1 =  \frac{\delta_1}{6c \chi_1} = \mathrm{max}(-\frac{c}{6c\chi_1},z_1(\chi_1,\chi_2,\chi_3)) = \mathrm{max}(\frac{-1}{6\chi_1},\frac{-1 }{6 \chi_2} - \frac{\chi_3}{9c^2\chi_1 \chi_2})
\end{equation}

Suppose we let $w = 1/24$ and $z = 9c^2/(24)^2 < 9(2/3)^2/(24)^2 < 1$.  Then $z \le 3wc^2 = c^2/8$, so all conditions on $c,w$ and $z$ are satisfied.  Suppose  $\chi_1, \chi_2$ are in the interval $[1/24,1/12]$ and $\chi_3$ is in the interval $[0, z]$.  Then
$$\frac{\sqrt{p}}{4 d_1} \delta_2  =  \mathrm{min}(\frac{1}{6\chi_1},\frac{1}{6 \chi_2} - \frac{\chi_3}{9c^2\chi_1 \chi_2}) \ge \mathrm{min}(2,2 - 1) \ge 1.$$
Since $\frac{\sqrt{p}}{4 d_1} \delta_1 < 0$, we will then have
$$\gamma(\mathcal(E)) \ge \frac{\sqrt{p}(\delta_2 - \delta_1)}{4 d_1} > 1.$$
Suppose $I_1$ and $I_2$ are intervals on the non-negative real axis of lengths $q_1, q_2 > 0$.  
By a sieving argument, as $r \to +\infty$, the number of coprime integers $(d_1,d_2)$ in a product $r I_1 \times rI_2$ is asymptotically $r^2 q_1 q_2 \prod_{\ell \ \mathrm{prime}} (1 - \ell^{-2}) = r^2 q_1 q_2 6/\pi^2$.  Applying this to $I_1 = [w3c\sqrt{p}/2,3c\sqrt{p}/2]$ and $I_2 = [0,3c\sqrt{p}/2]$ for  $w = 1/12$ as above, we see that as $p \to \infty$, the number of 
triples $(d_1,d_2,d_3) \in S(w,z)$ that have $d_1$ and $d_2$ coprime and  $\gamma(\mathcal{E}) > 1$ is bounded below by a positive constant times $\sqrt{p}^2\cdot p = p^2$.  Since the number of elements of $S(w,z)$ grows as a positive constant times $p^2$, Lemmas  \ref{lem:adjust} and \ref{lem:nicer} show that a positive fraction of all pairs $(t,a) \in (\mathbb{Z}/p)^* \times (\mathbb{Z}/p)^*$ lead to $\gamma(\mathcal{E}) > 1$.  

To prove that there are a positive proportion of pairs $(t,a)$ such that $\gamma(\mathcal{E}) = 0$, note that (\ref{eq:delt2comp}) and (\ref{eq:delt1comp}) give
$$\frac{\chi_1 \chi_2 \delta_2}{6c} \le  \frac{\chi_1}{6 } - \frac{\chi_3}{9c^2} \quad \mathrm{and}\quad \frac{ \chi_1 \chi_2 \delta_1}{6c} \ge -\frac{\chi_2}{6}$$
so 
$$\frac{\chi_1\chi_2}{6c} (\delta_2 - \delta_1) < \frac{\chi_1}{6 } - \frac{\chi_3}{9c^2}  + \frac{\chi_2}{6}.$$
The constraints on $\chi_1$, $\chi_2$ and $\chi_3$ are that $w \le \chi_1 \le 1$, $0 < \chi_2  \le 1$
and $0 \le \chi_3 \le z$, where $0 < w < 1$, $0 < 3c/2 < 1$, $0 \le z < 1$ and $z \le 3wc^2$. We now assume $0 < w < 1/2$.  Let $\chi_2 \to 0^+$, $\chi_1 \to w^+$, $\chi_3 \to z^-$ and $z = 3wc^2 < 3c^2/2 < c < 2/3$. Then
$\frac{\chi_1}{6 } - \frac{\chi_3}{9c^2}  + \frac{\chi_2}{6}$ has limit
$$\frac{w}{6} - \frac{z}{9c^2} + 0 = w \cdot (\frac{1}{6} - \frac{3c^2}{9 c^2}) =  - \frac{w}{6} < 0.$$
Thus taking $\chi_1$, $\chi_2$ and $\chi_3$ near these limits leads (via the same seiving argument used before) to a positive proportion of $(t,a) \in  (\mathbb{Z}/p)^* \times (\mathbb{Z}/p)^*$  for which $\delta_2 < \delta_1$.  
Lemma \ref{lem:capcomp}  shows $\gamma(\mathcal{E}) = 0$ for such $(t,a)$. 

Theorem \ref{thm:mainthm}, together with the above constructions of a positive proportion of $(t,a)$ with $\gamma(\mathcal{E}) > 1$ and of a positive proportion of $(t,a)$ for which  $\gamma(\mathcal{E}) = 0$ now proves Theorem \ref{thm:propexamples}.

\section{Bounds on the number of solutions of Problem \ref{prob:Integers}.}
\label{s:bounds}

We will be concerned with finding upper bounds on the number $N(t,a,\mathcal{J},X,Y)$ of pairs $(x,y)$ of algebraic integers having the properties in Problem \ref{prob:Integers}, where $N(t,a,\mathcal{J},X,Y)$  may be infinite.   This is relevant to the following case of the hidden number problem.

Suppose we are given two pairs $(a_1,b_1)$
 and $(a_2, b_2)$ of elements of $O_F$ such that for an unknown secret $s \in O_F$, one has $b_i = sa_i + e_i$ mod $\mathcal{J}$ for a small error $e_i \in O_F$.  Then $e_2  = b_2 - s  a_2 = b_2 - (b_1 - e_1) a_1^{-1} a_2$ mod $\mathcal{J}$ .  Thus if we let $x = e_2$, $y = e_1$, $t = a_1^{-1} a_2$ mod $\mathcal{J}$
 and $a = -b_2 + b_1 a_1^{-1} a_2$ mod $\mathcal{J}$, we wil have $x + ty  + a \equiv 0$ mod $\mathcal{J}$ with $x$ and $y$ small.  Therefore
$N(t,a,\mathcal{J},X,Y)$ gives a bound on the number of secrets $s$ which can solve this case of the hidden number problem.

To bound $N(t,a,\mathcal{J},X,Y)$,
it is simplest to deal with the case $a = 0$. This gives an upper bound for arbitrary $a$ at the cost of halving the allowed sizes of archimedean conjugates.

\begin{thm}
\label{thm:bouncer}  The following is true for all $t$, $\mathcal{J}$, $X$ and $Y$.
\begin{enumerate}
\item[1.]  When $a = 0$, there are either no $(x, y)$ with the properties in  Problem \ref{prob:Integers} or  there are infinitely many such $(x,y)$.
\item[2.]   Suppose $a = 0$ and $\gamma(\mathcal{E}) < 1$
in Theorem \ref{thm:mainthm}. Then there are no $(x,y)$ satisfying the conditions in Problem \ref{prob:Integers}, i.e. $N(t,0,\mathcal{J},X,Y) = 0$. 
\item[3.]  For all $a$ one has  $N(t,a,\mathcal{J},X/2,Y/2)  \le 1 + N(t,0,\mathcal{J},X,Y)$.  Thus either $N(t,0,\mathcal{J},X,Y) = \infty$ or $N(t,0,\mathcal{J},X,Y) = 0$ and $N(t,a,\mathcal{J},X/2,Y/2) \le 1$.
\end{enumerate}
\end{thm}

\begin{proof}  For part (1), observe that if $(x,y) \in \overline{\mathbb{Z}} \times \overline{\mathbb{Z}}$ has the properties in Problem \ref{prob:Integers} when $a = 0$, then
so does $(\zeta x, \zeta y)$ for any root of unity $\zeta$.  To prove (2), note that Theorem \ref{thm:mainthm} shows $N(t,0,\mathcal{J},X,Y)$ is finite if $\gamma(\mathcal{E}) < 1$. Then part (1) forces  $N(t,0,\mathcal{J},X,Y) = 0$.  Part (3) follows from the fact that the difference of two solutions $(x,y)$ and $(x',y')$ to Problem \ref{prob:Integers} for given $t$, $a$, $\mathcal{J}$, $X/2$ and $Y/2$ is a solution $(x'',y'') = (x' - x, y' - y)$ to 
Problem \ref{prob:Integers} for $t$, $0$, $\mathcal{J}$, $X$ and $Y$.
\end{proof}

\begin{rem} The proof of parts (1) and (2)  illustrates the advantages of working over the ring $\overline{\mathbb{Z}}$ of all algebraic integers, rather than in the integers of a particular number field.  This makes it possible to promote a finiteness result coming from capacity theory to a proof that a homogenous linear congruence has no small solutions at all. 
\end{rem}

We illustrate this result with a concrete application to the hidden number problem. Suppose as in \S \ref{s:hiddennumber} that we  are given an ideal $\mathcal{J}$
of the integers $O_F$ of a number field $F$ and a real number $X$.  For a secret integer $s \in O_F$ we are given samples $(c_i,d_i) \in O_F \times O_F$ for $i = 0,1$ for which $c_0$ is prime to $\mathcal{J}$ and the following is true.
There is an (unknown) element $x_i \in O_F$ such that 
\begin{equation}
\label{eq:congruences}
c_i s - d_i  \equiv x_i \quad \mathrm{mod}\quad \mathcal{J}
\end{equation}
and $|\lambda(x_i)| \le X/2$ for all embeddings $\lambda:F \to \mathbb{C}$.  We would like to determine $s$ mod $\mathcal{J}$ from this information.  Theorem \ref{thm:bouncer} leads in the following way to a computable criterion for their to exist at most one solution $s$ mod $\mathcal{J}$.

As in \S \ref{s:hiddennumber}, we find $c'_0 \in O_F$ such that $c_0 c'_0 \equiv 1 $ mod $\mathcal{J}$.   Then (\ref{eq:congruences}) for $i = 0$
gives
$$s \equiv c'_0 (x_0 + d_0) \quad \mathrm{mod}\quad \mathcal{J}.$$
Substituting this into (\ref{eq:congruences}) when $i = 1$ gives 
\begin{equation}
\label{eq:urky}
x_1 + t_i x_0 + a_1 \equiv 0 \quad \mathrm{mod}\quad \mathcal{J}
\end{equation}
where $t_1 = -c_1 c'_0$ and $a_1 = d_1 - d+ c_1 c'_0 d_0$. Thus the problem of finding all $s$ mod $\mathcal{J}$ satsifying the above conditions
is converted to finding all solutions $(x_0,x_1) \in O_F\times O_F$ of the congruence (\ref{eq:urky}) such that $|\lambda(x_i)| < X/2$ for $i = 0, 1$ and all embeddings
$\lambda : F \to \mathbb{C}$.  

We bound the number of  solutions $s$ mod $\mathcal{J}$ in the following way.   Using lattice basis reduction, find a polynomial
$b_1 x + b_2 y + b_3 \in \mathcal{J}^{-1} O_F[x,y]$ with the properties in Theorem \ref{thm:lin} for $Y = X$, $t = -c_1 c'_0$ and $a = 0$.  Calculate the capacity $\gamma(\mathcal{E})$ of the adelic set $\mathcal{E}$ associated to this adelic set in Definition \ref{def:definitions}, using Lemma \ref{lem:answers}
and Theorem \ref{thm:explicit}.  If $\gamma(\mathcal{E}) < 1$, then parts (2) and (3) of Theorem \ref{thm:bouncer} show $N(t,a,\mathcal{J},X/2,X/2) \le 1$.
Thus there is at most one pair $(x_0,x_1)$ as above, and at most one integer $s$ mod $\mathcal{J}$ which solves the above case of the hidden number problem.

We conclude this paper with another example illustrating Theorem \ref{thm:bouncer}. Suppose $a = 0$, $t \in O_F$ and that $\mathcal{J} = O_F \alpha$ is a non-zero principal ideal of $O_F$.  Suppose $(x_0,y_0) \in O_F$ satisfy the congruence
\begin{equation}
\label{eq:nicerhyp}
x_0 + t y_0 \equiv 0 \quad \mathrm{mod}\quad \mathcal{J}
\end{equation}
and that 
\begin{equation}
\label{eq:nicerbound}
|x_0 \cdot y_0|_v \le |\alpha|_v/2 \quad \mathrm{for \ all}\quad v \in \Minf.
\end{equation}
Suppose as before that $t$ is prime to $\mathcal{J} = O_F \alpha$, and that $x_0$, $y_0$ and $\alpha$ are pairwise relatively prime, in the sense that the ideal generated by any two of them is $O_F$.  

\begin{thm}
\label{thm:surprise} With the above hypotheses, there are no non-zero pairs $(x,y) \in \overline{\mathbb{Z}} \times \overline{\mathbb{Z}}$ with the following properties:
\begin{enumerate}
\item[1.] $x + ty \equiv 0$ mod $ \mathcal{J} \cdot \overline{\mathbb{Z}}$, and
\item[2.]  For all embeddings $\lambda:\overline{\mathbb{Z}} \to \mathbb{C}$, one has
$$|\lambda(x)| \le |\lambda(x_0)|\quad \mathrm{and}\quad |\lambda(y)| \le |\lambda(y_0)|$$
with at least one of these inequalities being strict for at least one $\lambda$.  
\end{enumerate}
\end{thm}

Thus in a small non-zero solution of the homogenous congruence resulting from setting $a = 0$ prevents the existence of non-trivial solutions with smaller archimedean absolute values.  Concerning the relation of the inequality (\ref{eq:nicerbound}) to Problem \ref{prob:Integers}, note
that if $0 < \alpha \in \mathbb{Z}$, then  (\ref{eq:nicerbound}) follows from requiring $|x_0|_v \le X$ and $|y_0| \le Y$ for some real $X, Y$ such that $|XY| \le \alpha/2$.  
\medbreak

\noindent {\bf Proof of Theorem  \ref{thm:surprise}}
\medbreak
 Define a polynomial in the variables $x$ and $y$ by 
\begin{equation}
\label{eq:bdef}
b_1 x + b_2 y = (y_0 x - x_0 y)/\alpha.
\end{equation}
Hypothesis (\ref{eq:nicerhyp}) shows $x_0 + t y_0 = \alpha z_0$ for some $z_0 \in O_F$.   Therefore
\begin{equation}
\label{eq:congruential}
b_1 x + b_2 y = y_0(x + ty)/\alpha - z_0 y \in \mathcal{J}^{-1} \cdot (x + ty) + O_F \cdot y.
\end{equation}
We now substitute for the variables $x$ and $y$ a pair of elements of $\overline{\mathbb{Z}}$ satisfying conditions (1) and (2) of Theorem \ref{thm:surprise}.  
The inequalities in condition (2) of Theorem \ref{thm:surprise} show that for each embedding $\lambda:\overline{\mathbb{Z}} \to \mathbb{C}$ we have
\begin{equation}
\label{eq:upper}
|\lambda(b_1 x + b_2 y)| = |\lambda(\frac{y_0 x - x_0 y}{\alpha})| \le |\frac{\lambda(y_0) \lambda(x)}{\alpha}| + |\frac{\lambda(x_0) \lambda(y)}{\alpha}|.
\end{equation}
Now (\ref{eq:nicerbound}) gives $$|\frac{\lambda(y_0)}{\alpha}| \le \frac{1}{2|\lambda(x_0)|} \quad \mathrm{and}\quad \frac{|\lambda(x_0)|}{\alpha} \le \frac{1}{2|\lambda(y_0)|}.$$
Substituting
this into (\ref{eq:upper}) gives
$$|\lambda(b_1 x + b_2 y)| \le \frac{|\lambda(x)|}{2|\lambda(x_0)|} + \frac{|\lambda(y)|}{2|\lambda(y_0)|}$$
Hypothesis (2) of Theorem \ref{thm:surprise} now implies $|\lambda(b_1 x + b_2 y)| \le 1$ with strict equality for at least one $\lambda$.  Since  $b_1 x + b_2 y$ is an algebraic integer, we conclude that
\begin{equation}
\label{eq:zeroed}
b_1 x + b_2 y  = 0 \quad \mathrm{when} \quad b_1 = y_0/\alpha \quad \mathrm{and} \quad b_2 = - x_0/\alpha.
\end{equation}
We enlarge $F$ so that it includes $x$ and $y$. There is then an archimedean place $v_\infty$ of $F$ at which either $|x|_{v_\infty} <  |x_0|_{v_\infty}$ or $|y|_{v_\infty} < |y_0|_{v_\infty}$ .  
For simplicity we will suppose that $r_{v_\infty} = |y|_{v_\infty}|/ |y_0|_{v_\infty} < 1$, the other case being similar. Define $r_v = 1$ if $v_\infty \ne v \in \Minf$.  

We now define an adelic set $\mathcal{E} = \prod_{v \in M_F} E_v$ associated to $b_1 x + b_2 y$ in the following way.  Set $b_3 = 0$ in Definition \ref{def:definitions}.
If $v \in \Mfin$ is finite, let $E_v$ be as in part (i) of Definition (\ref{def:definitions}). If $v \in \Minf$ is an infinite place, define 
\begin{equation}
\label{eq:archdef}
E_v = \{ y \in \overline{F}_v:  |y|_v \le r_v |y_0|_v \quad \mathrm{and} \quad  |x|_v = |b_2 y/b_1|_v = |-x_0 y / y_0|_v \le  |x_0|_v\} = \{ y \in \overline{F}_v:  |y|_v \le r_v |y_0|_v\}
\end{equation}
where we have used $r_v \le 1$.

 As in the proof of part (2) of Theorem \ref{thm:mainthm}, if $\gamma(\mathcal{E}) < 1$,
then there will be only  finitely many pairs $(x,y) \in \overline{\mathbb{Z}} \times \overline{\mathbb{Z}}$ satisfying the conditions in Theorem \ref{thm:surprise} and for which
$|y|_v \le r_v |y_0|_v$.   Then Theorem \ref{thm:bouncer} shows that in fact there are no such pairs, contradicting the hypothesis above that such a pair exists.  We conclude that to prove Theorem \ref{thm:surprise} if will suffice to show $\gamma(\mathcal{E}) < 1$.

We first need to describe explicitly the set $E_v$ when $v \in \Mfin$. From Definition \ref{def:definitions} and (\ref{eq:bdef}) we see that $E_v$ is the set of
$y \in \overline{F}_v$ satisfying 
\begin{enumerate}
\item[i.]  $|y|_v \le 1$
\item[ii.] $|-x_0 y /y_0|_v \le 1$ and 
\item[iii.] $|x_0y /y_0 +ty|_v \le |\mathcal{J}|_v = |\alpha|_v$.
\end{enumerate}
  Let us first show (i) and (ii) imply (iii).  If $|\alpha|_v  = 1$, this is clear from $t \in O_F$.  Suppose now that $|\alpha|_v  < 1$.  We know $y_0$ is prime to $\alpha$, so $|y_0|_v = 1$.  We  have  $(y_0,-x_0) \equiv y_0(1, t)$ mod $\mathcal{J}$ by multiplying the first equality in (\ref{eq:congruential}) by $\alpha$, so $|x_0 + t y_0|_v \le |\mathcal{J}|_v$.  Thus $$|(x_0/y_0 + t) y|_v = |y_0|_v^{-1} \cdot |x_0 + t y_0|_v \cdot |y|_v \le |\mathcal{J}|_v \quad 
\mathrm{if}\quad  |y|_v \le 1.$$
Therefore (iii) is implied by (i) when $|\alpha|_v$.  Thus $E_v$
is the set of $y \in \overline{F}_v$  satisfying (i) and (ii).  

Recall now that $y_0$ and $-x_0$ are coprime elements of $O_F$ by assumption.  Thus if $|y_0|_v < 1$ then $|-x_0|_v = 1$, and otherwise $|y_0|_v = 1$.  If $|y_0|_v = 1$,
then (i) implies $|-x_0 y /y_0|_v = |-x_0 y |_v \le 1$, so (ii) holds.  If $|y_0|_v < 1$, then $|-x_0|_v = 1$ so (ii) is equivalent to 
$|y |_v \le |y_0/(-x_0)|_v = |y_0|_v < 1$.  Hence condition (ii) implies (i) if $|y_0|_v < 1$.  We thus
find that for all finite $v$, $E_v$ is a $v$-adic disc of radius $r_v = |y_0|_v$ around  $0 \in F_v$.  The local capacity of $E_v$ is therefore
\begin{equation}
\label{eq:localfinite}
\gamma_v(E_v) = |y_0|_v^{[F_v:\mathbb{Q}_{p(v)}]} = ||y_0||_v \quad \mathrm{for} \quad v \in \Mfin
\end{equation}
where $p(v)$ is the residue characteristic of $v$ and $|| \ ||_v$ is the normalized valuation at $v$.  

We now consider archimedean $v \in M_{F,\inf}$.  From (\ref{eq:archdef}) we see that $E_v$ is the closed disc around $0$ in $\overline{F}_v = \mathbb{C}$
of radius $r_v  |y_0|_v$.  Thus the local capacity is
\begin{equation}
\label{eq:localinf}
\gamma_v(E_v) = (r_v |y_0|_v)^{[F_v:\mathbb{R}]} = r_v^{[F_v:\mathbb{R}]} ||y_0||_v \quad \mathrm{for} \quad v \in \Minf.
\end{equation}

Now (\ref{eq:localfinite}) and (\ref{eq:localinf}) together with the product formula give the global capacity of $\mathcal{E}$ as
$$\gamma(\mathcal{E}) = \prod_v \gamma_v(E_v) = \prod_{v \in \Minf} r_v^{[F_v:\mathbb{R}]} \cdot  \prod_{v\in M_F}  ||y_0||_v  =  \prod_{v \in \Minf} r_v^{[F_v:\mathbb{R}]}  < 1$$
which completes the proof of Theorem \ref{thm:surprise}.

\vfill
\eject

\bibliographystyle{plain} 
\bibliography{abbrev0,crypto,other-refs}

\end{document}